\documentclass[11pt, twoside]{article}
\usepackage{amssymb, amsmath, amsthm}
\usepackage[left=20mm, right=20mm, bottom=23mm, top=25mm]{geometry}
\usepackage{inputenc}
\usepackage{fancyhdr}
\usepackage{tikz}
\usetikzlibrary{positioning}
\usepackage{titling}
\usepackage{hyperref}
\predate{}
\postdate{}
\usepackage{lipsum}
\usepackage{titling}
\usepackage{xcolor}
\usepackage{placeins}
\usepackage{tikz}
\usepackage{verbatim}
\usepackage{titling}
\usepackage[T1]{fontenc}
\usepackage{currvita}
\usepackage{bbm}
\usepackage{enumitem}
\usepackage{cleveref}
\newtheorem{theorem}{Theorem}
\newtheorem{lemma}[theorem]{Lemma}
\newtheorem{proposition}[theorem]{Proposition}

\newtheorem{claim1}{Claim}

\newtheorem{conjecture}[theorem]{Conjecture}

\theoremstyle{definition}

\theoremstyle{remark}

\usepackage{xcolor}

\begin{document}

\newcommand{\Addresses}{
\bigskip
\footnotesize

\medskip

\noindent Maria-Romina~Ivan, \textsc{Department of Pure Mathematics and Mathematical Statistics, Centre for Mathematical Sciences, Wilberforce Road, Cambridge, CB3 0WB, UK.}\par\noindent\nopagebreak\textit{Email addresses: }\texttt{mri25@dpmms.cam.ac.uk}

\medskip

\noindent Sean~Jaffe, \textsc{Department of Pure Mathematics and Mathematical Statistics, Centre for Mathematical Sciences, Wilberforce Road, Cambridge, CB3 0WB, UK.}\par\noindent\nopagebreak\textit{Email address: }\texttt{scj47@cam.ac.uk}}

\pagestyle{fancy}
\fancyhf{}
\fancyhead [LE, RO] {\thepage}
\fancyhead [CE] {MARIA-ROMINA IVAN AND SEAN JAFFE}
\fancyhead [CO] {THE BOUNDED-LINEAR DICHOTOMY HOLDS FOR ALL 4 POINT POSETS}
\renewcommand{\headrulewidth}{0pt}
\renewcommand{\l}{\rule{6em}{1pt}\ }
\title{\Large{\textbf{THE BOUNDED-LINEAR DICHOTOMY HOLDS FOR ALL 4 POINT POSETS}}}
\author{MARIA-ROMINA IVAN AND SEAN JAFFE}
\date{ }
\maketitle
\begin{abstract}
Given a finite poset $\mathcal P$, its induced saturation number, $\text{sat}^*(n,\mathcal P)$ is the smallest size a family of subsets of $[n]$ can have such that it does not contain an induced copy of $\mathcal P$, but adding any other set to it creates one. The saturation numbers already show a sharp dichotomy -- Freschi, Piga, Sharifzadeh and Treglown showed that for any given poset its saturation number is either bounded or at least $2\sqrt n$. The dominant conjecture is that in fact, the saturation numbers are either bounded, or exactly linear.

In this paper we show that this dichotomy is true for all posets on at most 4 points. Most of these posets, most notably the butterfly and the diamond, had individually long resisted analysis, and the missing piece which this paper tackles is the poset comprised of a $\mathcal V$ and an isolated point, denoted by $\widehat{\mathcal V}$, for which we show linear saturation.

The architecture of the proof uncovers the structure of an arbitrary $\widehat{\mathcal V}$- saturated family, namely an antichain with at most two chains above it, and an upside-down forest below it. Moreover, we anticipate that, at least partially, this structure is relevant to the wider class of posets that contain an isolated point. 
\end{abstract}
\section{Introduction}
We say that a poset $(\mathcal Q,\preceq)$ contains an \textit{induced copy} of a poset $(\mathcal P,\preceq')$ if there exists an injective order-preserving function $f:\mathcal P\rightarrow\mathcal Q$ such that $(f(\mathcal P),\preceq)$ is isomorphic to $(\mathcal P,\preceq')$. We denote by $\mathcal P([n])$ the power set of $[n]=\{1,2,\dots,n\}$. Moreover, for any finite set $S$, we denote by $\mathcal P(S)$ the power set of $S$. 

If $\mathcal P$ is a finite poset and $\mathcal F$ is a family of subsets of $[n]$, we say that $\mathcal F$ is $\mathcal P$-\textit{saturated} if $\mathcal F$ does not contain an induced copy of $\mathcal P$, and for any $S\notin\mathcal F$, the family $\mathcal F\cup\{S\}$ contains an induced copy of $\mathcal P$. The smallest size of a $\mathcal P$-saturated family of subsets of $[n]$ is called the \textit{induced saturation number}, denoted by $\text{sat}^*(n,\mathcal P)$. For the rest of the paper all posets are finite, and all copies are induced.

It has been shown that the growth of $\text{sat}^*(n,\mathcal P)$ has a dichotomy. Keszegh, Lemons, Martin, P{\'a}lv{\"o}lgyi and Patk{\'o}s \cite{keszegh2021induced} proved that for any poset the induced saturation number is either bounded or at least $\log_2(n)$. They also made the following conjecture.
\begin{conjecture}[\cite{keszegh2021induced}]
    \label{conjecture:main} Let $\mathcal P$ be a finite poset. Then either $\text{sat}^*(n,\mathcal P)$ is bounded, or $\text{sat}^*(n,\mathcal P)\geq n+1$.
\end{conjecture}

Later, Freschi, Piga, Sharifzadeh and Treglown \cite{freschi2023induced} improved this result by replacing $\log_2 (n)$ with $2\sqrt{n}$. 

In the other direction, there is no known poset $\mathcal P$ for which $\text{sat}^*(n,\mathcal P)=\omega(n)$, and it is in fact believed that for any poset, the saturation number is either constant or it grows linearly.

\vspace{0.5em}
We digress slightly and discuss briefly what forces a poset to have unbounded saturation number.

We say that a family $\mathcal F\subseteq\mathcal P([n])$ \textit{separates the ground set} if, for any two distinct elements $i,j\in[n]$, there exists a set in $\mathcal F$ that contains one but not the other. It was shown in \cite{keszegh2021induced,ferrara2017saturation} that the existence of an $n_0$ and a $\mathcal P$-saturated family in $\mathcal P([n_0])$ that does not separate the elements of the ground set implies that $\mathcal P$ has bounded saturation number. Conversely, they showed that if a $\mathcal P$-saturated family separates the ground set $[n]$, it has size at least $\log_2(n)$. Later, Freschi, Piga, Sharifzadeh and Treglown \cite{freschi2023induced} modified the notion of separability to `for every $i\in[n]$ there exist $A,B\in\mathcal F$ such that $A\setminus B=\{i\}$', and arrived at the same conclusion, with $\log_2(n)$ replaced by $2\sqrt n$.

However, this tells us nothing about the intrinsic properties a poset must have in order for such a saturated family to exist. In fact, this question is very much open.

Partial progress towards answering this question was made in \cite{ferrara2017saturation} where it was shown that if a poset has the \textit{unique cover twin property (UCTP)}, then it has unbounded saturation number. An extension of this result was proved in \cite{keszegh2021induced}, where it was shown that attaching a chain above any poset with the UCTP gives a poset with unbounded saturation number still. It was also shown by Freschi, Piga, Sharifzadeh and Treglown \cite{freschi2023induced} that any poset with \textit{legs} (elements $l_1, l_2 $ such that $l_1, l_2$ are incomparable and less than $h$, and any other element of the poset is greater than $h$) has a saturation number that is at least linear, and therefore unbounded. This was further generalised by Liu \cite{liu2025induced} who removed the intermediary point $h$. Most recently, Ivan and Jaffe generalised these results \cite{gluing} by showing that taking the \textit{linear sum} of a poset with the UCTP on at least 2 elements and any arbitrary poset results in a poset with unbounded saturation number.

Sadly, these conditions are not exhaustive as the poset $2C_2$, which does not have the UCTP (or UCTP with top chain), does not have legs, and is not a linear sum of a poset with the UCTP and another poset, was shown to have unbounded saturation number \cite{keszegh2021induced}.
\vspace{0.5em}
   
Returning to the question of determining the saturation number, there are very few posets for which even the growth order of the saturation number is known, and far fewer for which the exact number is known. Some posets for which it is known that the saturation number is linear include $\mathcal V_k$, $\Lambda_k$ for $k\geq2$ \cite{ferrara2017saturation, sums}, the butterfly \cite{ivan2020saturationbutterflyposet, keszegh2021induced}, the antichain \cite{keszegh2021induced, dhankovic2023saturation, bastide2024exact}, the $\mathcal N$ poset \cite{N}, and the diamond \cite{diamondlinear, exact}. It is worth mentioning that so far, each poset had to be analysed with tools developed for its own unique structure. In particular, part of the difficulty in the relatively young field of poset saturation is the lack of powerful general techniques.

In this paper we establish that all posets that have at most 4 elements have saturation number either linear or bounded. For the unique poset with 1 element, it is clear that the saturation number is 0. For 2 elements, there are exactly two posets (up to isomorphism), namely the chain $C_2$ and the antichain $\mathcal A_2$. Trivially, $\{\emptyset\}$ is $C_2$-saturated, hence $C_2$ has bounded saturation number. For $\mathcal A_2$, it is easy to see that any $\mathcal A_2$-saturated family has to be a full chain, and consequently $\text{sat}^*(n,\mathcal A_2)=n+1$, hence linear.

Moving on, there are exactly five posets on 3 points (up to isomorphism): the chain $C_3$, the antichain $A_3$, $\mathcal V$, $\Lambda$, and the poset comprised of a chain and an isolated element, which we denote by $C_2+C_1$. For $C_3$, it is trivial to see that the family $\{\emptyset, [n]\}$ is $C_3$-saturated, hence $C_3$ has bounded saturation number. Moreover, it was shown in \cite{keszegh2021induced} that the saturation number of $C_2+C_1$ is 4, hence also bounded. Next, \cite{ferrara2017saturation} established that the saturation number for both $\mathcal V$ and $\Lambda$ is exactly $n+1$, and that also the saturation number for $\mathcal A_3$ is $2n$, hence all linear.

Finally, up to isomorphism, there are exactly 16 posets on 4 points. Before we list them, we note that the saturation number of a poset is the same as the saturation number of its dual, i.e. the poset obtained from it by reversing all the relations. Out of the 16, 8 of them are self-dual, and their names and Hasse diagrams are depicted below.

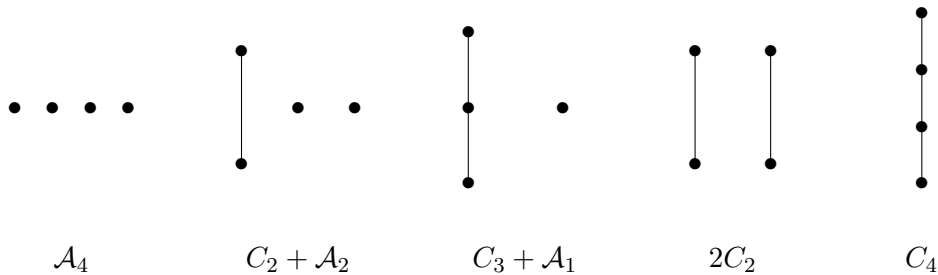
\begin{figure}[h]
\centering
\begin{tikzpicture}
	\node at (0,1) {$\bullet$};
	\node at (0.5,1) {$\bullet$};
	\node at (1,1) {$\bullet$};
	\node at (1.5,1) {$\bullet$};
	\node at (0.75,-1) {$\mathcal{A}_4$};
	
	\node at (3,0.25) {$\bullet$};
	\node at (3,1.75) {$\bullet$};
	\node at (3.75,1) {$\bullet$};
	\node at (4.5,1) {$\bullet$};
	\draw (3,0.25) -- (3,1.75);
	\node at (3.75,-1) {$C_2+\mathcal{A}_2$};
	
	\node at (6,0) {$\bullet$};
	\node at (6,1) {$\bullet$};
	\node at (6,2) {$\bullet$};
	\node at (7.25,1) {$\bullet$};
	\draw (6,0) -- (6,1) -- (6,2);
	\node at (6.75,-1) {$C_3+\mathcal{A}_1$};
	
	\node at (9,0.25) {$\bullet$};
	\node at (9,1.75) {$\bullet$};
	\node at (10,0.25) {$\bullet$};
	\node at (10,1.75) {$\bullet$};
	\draw (9,0.25) -- (9,1.75);
	\draw (10,0.25) -- (10,1.75);
	\node at (9.5, -1) {$2C_2$};
	
	\node at (12, 0) {$\bullet$};
	\node at (12, 0.75) {$\bullet$};
	\node at (12, 1.5) {$\bullet$};
	\node at (12, 2.25) {$\bullet$};
	\draw (12,0) -- (12,2.25);
	\node at (12,-1) {$C_4$};
\end{tikzpicture}
\end{figure}

\begin{figure}[h]
\centering
\begin{tikzpicture}
	\node at (0.5,-5) {$\bullet$};
	\node at (0.5,-3) {$\bullet$};
	\node at (2.5,-5) {$\bullet$};
	\node at (2.5,-3) {$\bullet$};
	\draw (0.5,-5) -- (0.5,-3) -- (2.5,-5) -- (2.5,-3);
	\node at (1.5,-6) {The $\mathcal{N}$ poset};
	
	\node at (6,-5) {$\bullet$};
	\node at (5,-4) {$\bullet$};
	\node at (7,-4) {$\bullet$};
	\node at (6,-3) {$\bullet$};
	\draw (6,-5) -- (5,-4) -- (6,-3) -- (7,-4) -- (6,-5);
	\node at (6,-6) {The diamond};
	
	\node at (9.5,-5) {$\bullet$};
	\node at (11.5,-5) {$\bullet$};
	\node at (9.5,-3) {$\bullet$};
	\node at (11.5,-3) {$\bullet$};
	\draw (9.5,-5) -- (9.5,-3) -- (11.5,-5) -- (11.5,-3) -- (9.5,-5);
	\node at (10.5,-6) {The butterfly};
\end{tikzpicture}
\end{figure}

As already mentioned above, we know that the saturation number for the $\mathcal N$ poset, the diamond, and the butterfly is linear. In \cite{keszegh2021induced} it was shown that the saturation numbers for both $C_2+\mathcal A_2$ and $C_3+\mathcal A_1$ are at most 8, hence bounded, as well as the fact that the saturation number for $2C_2$ is between $n+2$ and $2n$, hence linear. In \cite{ferrara2017saturation} it was shown that the saturation number for $\mathcal A_4$ is $3n-1$, so linear, and in \cite{gerbner2013saturating} it was shown that the saturation number for $C_4$ is 4, so bounded.

Finally, the remaining 8 posets come in  4 pairs: the poset and its dual. The Hasse diagrams are depicted below, where only one representative of each such pair is shown.

\begin{figure}[h]
	\centering
	\begin{tikzpicture}
		
		\node at (0,0) {$\bullet$};
		\node at (-1,2) {$\bullet$};
		\node at (0,2) {$\bullet$};
		\node at (1,2) {$\bullet$};
		\draw (0,0) -- (-1,2);
		\draw (0,0) -- (0,2);
		\draw (0,0) -- (1,2);
		\node at (0,-1) {$\mathcal{V}_3$};
		
		\node at (4,0) {$\bullet$};
		\node at (4,1) {$\bullet$};
		\node at (3,2) {$\bullet$};
		\node at (5,2) {$\bullet$};
		\draw (4,0) -- (4,1) -- (3,2);
		\draw (4,1) -- (5,2);
		\node at (4,-1) {The $Y$ poset};
		
		\node at (8,0) {$\bullet$};
		\node at (8.75,1) {$\bullet$};
		\node at (7.25,1) {$\bullet$};
		\node at (6.5,2) {$\bullet$};
		\draw (6.5,2) -- (8,0) -- (8.75,1);
		\node at (8,-1) {Diamond without an edge};
		
		\node at (12,0) {$\bullet$};
		\node at (11,1.5) {$\bullet$};
		\node at (13,1.5) {$\bullet$};
		\node at (14,0.75) {$\bullet$};
		\draw (11,1.5) -- (12,0) -- (13,1.5);
		\node at (12.75,-1) {The poset $\widehat{\mathcal{V}}$};
	\end{tikzpicture}
\end{figure}

As discussed above, the saturation number for any fork $\mathcal V_k$ with $k\geq2$ is linear. Furthermore, \cite{freschi2023induced} showed that the saturation number for the $Y$ poset is $n+2$, hence linear, and \cite{keszegh2021induced} showed that the saturation number for the diamond without an edge is 4, hence bounded. Therefore, the only remaining 4-point poset for which the bounded-linear dichotomy has not yet been established is $\widehat{\mathcal V}$ (and its dual). 

In this paper we establish that the saturation number for $\widehat{V}$ is linear, which completes the list. Since it is easy to find a $\widehat{V}$-saturated family of linear size (all complements of singletons together with $\emptyset$ and $[n]$), the paper focuses on establishing the linear lower bound. The proof strategy is as follows. We first observe that if $\mathcal F$ is a $\widehat{\mathcal V}$-saturated family such that all its sets are comparable to some fixed set (other than $\emptyset$ and $[n]$), then we can decompose $\mathcal F$ into two $\widehat{\mathcal V}$-saturated families with strictly smaller ground sets, which leads to an inductive argument. Thus, we assume that our saturated family $\mathcal F$ does not have this property -- it is `inseparable'. We then choose a maximal antichain with the smallest up-set in $\mathcal F$ and show that above this antichain our family is contained in at most 2 chains, whilst below it it is an upside-down forest. By choosing a special element $A$ of this antichain, we obtain the two main counting lemmas, one coming from the singletons inside $A$, while the other coming from the ones outside $A$. It turns out that for the singletons inside $A$ we can only `descend' up to a minimal element of $\mathcal F\setminus\{\emptyset\}$. However, a third, more straightforward counting argument lower-bounds the size of $\mathcal F$ by the size of such a minimal element, which is then enough to close the proof.

The plan of the paper is the following. In Section 2 we prove the most straightforward counting argument, Lemma~\ref{lemma:minimal}, which says that the size of $\mathcal F$ is at least 2 more than the size of a minimal element of $\mathcal F\setminus\{\emptyset\}$, as well as set up the above mentioned antichain and inseparable structure. In Section 3 we completely focus on the inseparable case (there exist no non-trivial set that is comparable to all sets of $\mathcal F$) and prove linearity directly by combining Lemmas~\ref{lemma:count1} and~\ref{lemma:count2} (counting above and below a certain element of the antichain) with Lemma~\ref{lemma:minimal}. Section 4 ties all previous sections together and establishes the claimed linearity for $\widehat{\mathcal V}$.

\section{Preliminaries}
Let $\mathcal{F}\subseteq\mathcal P([n])$ be a $\widehat{\mathcal{V}}$-saturated family such that $|\mathcal F|<2n$, and $n\geq3$. Then, it is clear that $\emptyset\in\mathcal F$ and $[n]\in\mathcal{F}$, as they can never be part of an induced copy of $\widehat{\mathcal V}$. We now prove the following standard lemma.

\begin{lemma}\label{lemma:minimal}
Let $M$ be a minimal element of $\mathcal{F}\setminus\{\emptyset\}$. Then $|\mathcal{F}|\geq|M|+2$.
\end{lemma}

\begin{proof}
First, we show that for every $i\in M$, we can find $Q_i\in\mathcal{F}$ such that $M\setminus Q_i = \{i\}$.

We observe that when $n\geq3$, $\mathcal P([n])$ contains a copy of $\widehat{\mathcal V}$, thus $\mathcal F\neq\mathcal P([n])$. Therefore $|\mathcal{F}|\geq |\widehat{\mathcal{V}}|-1 = 3$, since if $X\not\in \mathcal{F}$, then we require $ |\widehat{\mathcal{V}}|-1$ many elements of $\mathcal F$ to form a copy of \(\widehat{\mathcal{V}}\) with $X$. If $M$ is a singleton, then we trivially have the claimed inequality, so we may assume that $|M|\geq 2$.

Fix $i\in M$. Then $M\setminus \{i\}\not\in \mathcal{F}$ as $M$ is minimal. Thus, adding $M\setminus\{i\}$ to $\mathcal{F}$ forms a $\widehat{\mathcal{V}}$.

\noindent\textbf{Case 1.} $M\setminus\{i\}$ is at the top of the $\mathcal{V}$. 

This case is depicted below, where $A,B\in\mathcal F$. This cannot happen as $M$ is minimal in $\mathcal F\setminus\{\emptyset\}$, $A\subsetneq M$, and $A\neq\emptyset$ as $A$ and $B$ are incomparable.

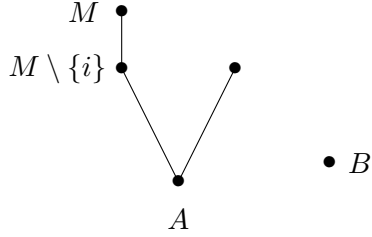
\begin{figure}[h]
\centering
\begin{tikzpicture}
\node at (0,0) {$A$};
\node at (0,0.5) {$\bullet$};
\draw (0,0.5) -- (0.75,2);
\draw (0,0.5) -- (-0.75,2);
\node at (0.75,2) {$\bullet$};
\node at (-0.75,2) {$\bullet$};
\node at (2,0.75) {$\bullet$};
\node at (2.4,0.75) {$B$};
\node at (-1.6,2) {$M\setminus\{i\}$};
\draw (-0.75,2) -- (-0.75,2.75);
\node at (-0.75,2.75) {$\bullet$};
\node at (-1.25,2.75) {$M$};
\end{tikzpicture}
\end{figure}

\noindent\textbf{Case 2.} $M\setminus \{i\}$ is the isolated point. 

Let $A,B,C\in\mathcal F$ form a $\mathcal V$ that is incomparable to $M\setminus\{i\}$, as depicted below.

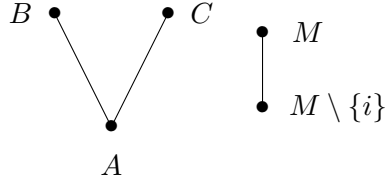
\begin{figure}[h]
\centering
\begin{tikzpicture}
\node at (0,0) {$A$};
\node at (0,0.5) {$\bullet$};
\draw (0,0.5) -- (0.75,2);
\draw (0,0.5) -- (-0.75,2);
\node at (0.75,2) {$\bullet$};
\node at (-0.75,2) {$\bullet$};
\node at (2,0.75) {$\bullet$};
\node at (3,0.75) {$M\setminus \{i\}$};
\node at (-1.2,2) {$B$};
\node at (1.2,2) {$C$};
\draw (2,0.75) -- (2,1.75);
\node at (2,1.75) {$\bullet$};
\node at (2.6, 1.75) {$M$};
\end{tikzpicture}
\end{figure}
Since $\mathcal F$ is $\widehat{\mathcal V}$-free, $M$ must be comparable to at least one of $A,B,C$. Since all of them are incomparable to $M\setminus\{i\}$ (and hence not $\emptyset$), at least one of them has to be a strict subset of $M$, contradiction.

\noindent\textbf{Case 3.} $M\setminus \{i\}$ is the minimal element of the $\mathcal{V}$.

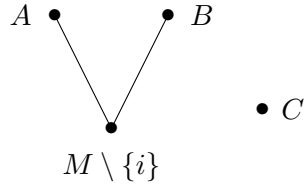
\begin{figure}[h]
\centering
\begin{tikzpicture}
\node at (0,0) {$M\setminus\{i\}$};
\node at (0,0.5) {$\bullet$};
\draw (0,0.5) -- (0.75,2);
\draw (0,0.5) -- (-0.75,2);
\node at (0.75,2) {$\bullet$};
\node at (-0.75,2) {$\bullet$};
\node at (2,0.75) {$\bullet$};
\node at (2.4,0.75) {$C$};
\node at (-1.2,2) {$A$};
\node at (1.2,2) {$B$};
\end{tikzpicture}
\end{figure}

Let $C$ be the isolated element of the copy of $\widehat{\mathcal{V}}$, and $A,B$ be the maximal elements of the copy. Of course $A,B,C\in\mathcal F$.  

If one of the top elements of the $\mathcal{V}$ (without loss of generality, say $A$) is equal to $M$, then $A=M$ is incomparable to $B$, which implies that $i\not\in B$, thus $M\setminus B = \{i\}$ as desired.

Therefore, we may assume that $A\neq M$ and $B\neq M$. If both $A$ and $B$ are comparable to $M$, we have the following diagram.

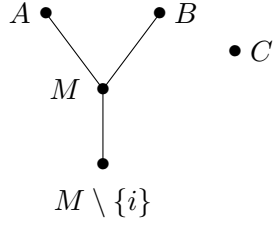
\begin{figure}[h]
\centering
\begin{tikzpicture}
\node at (0,0) {$\bullet$};
\node at (0,-0.5) {$M\setminus\{i\}$};
\draw (0,0) -- (0,1);
\node at (0,1) {$\bullet$};
\draw (0,1) -- (0.75,2);
\draw (0,1) -- (-0.75,2);
\node at (0.75,2) {$\bullet$};
\node at (-0.75,2) {$\bullet$};
\node at (-1.1,2) {$A$};
\node at (1.1,2) {$B$};
\node at (-0.5,1) {$M$};
\node at (1.75,1.5) {$\bullet$};
\node at (2.1,1.5) {$C$};
\end{tikzpicture}
\end{figure}

Since $A,B,M,C$ cannot form a $\widehat{\mathcal{V}}$ in $\mathcal{F}$, $M$ and $C$ must be comparable. However, this is not possible as $M\subseteq C$ implies $M\setminus\{i\}\subset C$, while $C\subseteq M$ implies $C\subset B$, a contradiction in both cases.
    
Therefore, one of $A$ or $B$ must be incomparable to $M$. Without loss of generality, we may assume that $A$ and $M$ are incomparable, which implies that $i\not\in A$, thus $M\setminus A = \{i\}$, as desired.

For the count, we note that no $Q_i$ can be $\emptyset$, thus all the $Q_i$, together with $M$ and $\emptyset$ give $|M|+2$ elements in $\mathcal F$, as desired.
\end{proof}

Next, we observe that $\mathcal F$ contains an antichain of size at least 3. Indeed, as noted in Lemma~\ref{lemma:minimal}, $\mathcal F\neq\mathcal P([n])$, thus there exists $X\notin\mathcal F$. By saturation, $\mathcal F\cup\{X\}$ contains a copy of $\widehat{\mathcal V}$, and consequently an antichain of size 3. Furthermore, the above argument shows that if $\mathcal F$ does not contain an antichain of size 3, then $\mathcal F$ is actually $\mathcal A_3$-saturated, which implies that $|\mathcal{F}|\geq 2n$ \cite{ferrara2017saturation}, a contradiction. Therefore $\mathcal F$ contains an antichain of size at least 3.

Let $\mathcal B\subseteq\mathcal F$ be a maximal antichain of size at least 3. For any such $\mathcal{B}$, let $\mathcal{U_{\mathcal B}}$ be the set $\{X\in\mathcal{F}:B\subsetneq X\text{ for some }B\in\mathcal{B}\}$. We call $\mathcal{U_{\mathcal B}}$ the \textit{upper cone} of $\mathcal{B}$.

Let $\mathcal{A}\subseteq\mathcal F$ be a maximal antichain of size at least 3 such that $|\mathcal{U_{\mathcal A}}|$ is minimal. For simplicity, from now on, we will simply use $\mathcal U$ for $\mathcal U_{\mathcal A}$.

\begin{lemma}\label{lemma:antichain}
Suppose that $\mathcal{C}\subseteq\mathcal{A}\cup \mathcal{U}$ is an antichain of size at least $3$. Then $\mathcal{C}\subseteq \mathcal{A}$.
\end{lemma}

\begin{proof}
Suppose that $\mathcal{C}\not \subseteq \mathcal{A}$. Extend $\mathcal{C}$ to a maximal antichain $\mathcal{C}^*$ inside $\mathcal{A}\cup\mathcal{U}$. Since $\mathcal{C}^*$ is maximal, every element of $\mathcal A$ is a subset (not necessarily strict) of an element of $\mathcal C^*$. If $\mathcal{C}^*$ is not a maximal antichain in $\mathcal F$, then there exists $X\in\mathcal{F}\backslash\left(\mathcal{A}\cup\mathcal{U}\right)$ such that $X$ is incomparable to every element of $\mathcal{C}^*$. Since $\mathcal{A}$ is maximal, this implies that there exists $A\in\mathcal{A}$ such that $X$ and $A$ are comparable. As $X\not\in \mathcal{A}\cup\mathcal{U}$, we must have $X\subset A$, which implies that $X$ is a subset of an element of $\mathcal C^*$, a contradiction.

Therefore $\mathcal{C}^*$ is a maximal antichain of $\mathcal F$ of size at least 3. Furthermore, since $\mathcal C^*$ is inside $\mathcal{A}\cup\mathcal{U}$, $\mathcal U_{\mathcal C^*}\subseteq \mathcal{U}$. Furthermore, since there exists $U_0\in\mathcal{C}\cap\mathcal U$, the above inclusion is strict as $U_0\not\in \mathcal{U}_{\mathcal{C}^*}$. This contradicts the minimality of $|\mathcal U|$, hence proving the lemma.
\end{proof}

Lemma~\ref{lemma:antichain} has the following implications that we will use throughout the paper. First of all, any antichain in $\mathcal{U}$ has size at most 2. Thus, by Dilworth's theorem, it can be covered by at most $2$ chains. Furthermore, every $U\in \mathcal{U}$ is incomparable to at most one member of $\mathcal{A}$. Finally, if $U_1,U_2\in\mathcal{U}$ are incomparable, then no member of $\mathcal{A}$ is incomparable to both.

Next, we introduce the following definitions which naturally give rise to a two-case analysis and an inductive argument.  

We call a set $\emptyset\subsetneq S\subsetneq [n]$, a \emph{separator} for $\mathcal{F}$ if $S$ is comparable to every set of $\mathcal{F}$. If such an $S$ exists, then $S\in\mathcal{F}$. Indeed, suppose $S\not \in \mathcal{F}$. It forms an induced copy of $\widehat{\mathcal{V}}$ when combined with $\mathcal{F}$. Since $S$ is incomparable to at least one element of that copy of $\widehat{\mathcal{V}}$, it must be incomparable to an element of $\mathcal{F}$. This is a contradiction, so $S\in \mathcal{F}$. If $\mathcal{F}$ does not contain a separator, we say that $\mathcal{F}$ is \emph{inseparable}, and call it \emph{separable} when it does.

Moreover, the set $\{X\in\mathcal{F}:X\subseteq S\}$ is a $\widehat{\mathcal{V}}$-saturated subset of $\mathcal{P}(S)$. Indeed, it does not contain a copy of $\widehat{\mathcal{V}}$ since it is a subset of $\mathcal F$, and if $X\in \mathcal{P}(S)\backslash\{X\in\mathcal{F}:X\subseteq S\} $, then $X\notin\mathcal F$, and consequently $\mathcal F\cup\{X\}$ contains a copy of $\widehat{\mathcal{V}}$ in which $X$ is one of the elements. If that copy contains an element $Y$ such that $S\subseteq Y$, then that $Y$ will contain $X$. Therefore $X$ is the minimal of the $\mathcal V$ and $Y$ is one of the maximals. Let $Z\in\mathcal F$ be the isolated point of the copy of $\widehat{\mathcal{V}}$. If $Z$ is below $S$, then $Z\subseteq Y$, a contradiction, whilst if $Z$ is above $S$, then $X\subseteq Z$, a contradiction. This is depicted below.

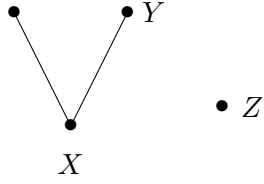
\begin{figure}[h]
\centering
\begin{tikzpicture}
\node at (0,0) {$X$};
\node at (0,0.5) {$\bullet$};
\draw (0,0.5) -- (0.75,2);
\draw (0,0.5) -- (-0.75,2);
\node at (0.75,2) {$\bullet$};
\node at (-0.75,2) {$\bullet$};
\node at (1.1,2) {$Y$};
\node at (2,0.75) {$\bullet$};
\node at (2.4,0.75) {$Z$};
\end{tikzpicture}
\end{figure}
Therefore all the sets involved in the copy must be below $S$, so $\{X\in\mathcal{F}:X\subseteq S\}$ is a $\widehat{\mathcal{V}}$-saturated. Thus, \(|\{X\in\mathcal{F}:X\subseteq S\}|\geq sat^*(|S|,\widehat{\mathcal{V}})\).

A completely analogous argument shows that the family $\{X\in\mathcal{F}, S\subseteq X\}$ is a $\widehat{\mathcal{V}}$-saturated subset of $\{X\in \mathcal{P}([n]) : X\supseteq S\}$. Since $\{X\in \mathcal{P}([n]) : X\supseteq S\}$ is isomorphic to \(\mathcal{P}([n]\backslash S)\), we have: $|\{X\in\mathcal{F}, S\subseteq X\}|\geq sat^*(n - |S|,\widehat{\mathcal{V}})$.

 Since the only element in both families is $S$, we get the following inequality:
\begin{equation*}
    |\mathcal{F}| \geq \text{sat}^*(|S|,\widehat{\mathcal{V}}) + \text{sat}^*(n-|S|,\widehat{\mathcal{V}})-1.
\end{equation*}

\section{The inseparable case}
In this section we assume that our saturated family $\mathcal F$ is not separable and obtain a linear lower bound directly. 
\begin{proposition}\label{inseparable}
Suppose that $\mathcal F$ is not separable. Then $|\mathcal F|\geq\frac{n}{2}+2$.  
\end{proposition}
\begin{proof} The proof relies on the global structure $\mathcal F$ must have, which is particularly explored in the next two lemmas. 
\begin{lemma}\label{lemma:size-3}
No element of $\mathcal F\setminus\{\emptyset\}$ is below at least two elements of an antichain in $\mathcal{F}$ of size at least 3.
\end{lemma}

\begin{proof}
Let $\mathcal{B}$ be an antichain in $\mathcal{F}$ of size at least $3$. Without loss of generality, we may assume that $\mathcal B$ is maximal in $\mathcal F$.

Suppose that there exists an element in $\mathcal{F} \backslash\{ \emptyset\}$ below  at least two sets of $\mathcal{B}$. Then this set has to be comparable to all elements of $\mathcal{B}$, otherwise we have a copy of $\widehat{\mathcal{V}}$ in $\mathcal{F}$. Since it is already below sets in $\mathcal B$, an antichain, it then follows that it has to be below all sets of $\mathcal B$.

Next, let $\mathcal{L}= \{X\in\mathcal{F}\backslash\{\emptyset\}:X\subseteq B\text{ for all }B\in\mathcal{B}\}.$ By the above observations, $\mathcal L$ is not empty. Now let $L= \bigcup_{X\in\mathcal{L}}X$. Since $\mathcal{L}$ is non-empty, we get that $L\neq\emptyset$. Also, $L\subseteq\bigcap_{B\in\mathcal B}B$, thus $L\neq[n]$.

Moreover, if $S\in\mathcal{F}$, then $S$ and $L$ are comparable. Indeed, by the maximality of $\mathcal B$, there exists $B\in\mathcal B$ such that $S$ and $B$ are comparable (not necessarily distinct). If $B\subseteq S$, then $X\subseteq S$ for all $X\in\mathcal{L}$, and so $L\subseteq S$. If $S\subsetneq B$, then $S$ cannot be above any element of $\mathcal{B}$ as $\mathcal B$ is an antichain. $S$ is either below every element of $\mathcal B$, or is below exactly one of them. In the former case, $S$ must be a subset of $L$ as $S\in \mathcal{L}\cup \{\emptyset\}$. In the latter case, $S$ is incomparable to every element of $\mathcal B\setminus\{B\}$. Suppose that $S$ and $L$ are incomparable. Then there exists $X\in\mathcal{L}$ such that $X\not\subseteq S$, and hence $X$ and $S$ are incomparable. Thus $X$, $S$, and any two members of $\mathcal{B}\setminus\{B\}$ form a copy of $\widehat{\mathcal{V}}$ in $\mathcal F$, a contradiction.

Therefore, every element of $\mathcal{F}$ is comparable to \(L\) and $L\notin\{\emptyset,[n\}]$, contradicting the fact that $\mathcal F$ is not separable. This completes the proof.
\end{proof}

{\color{red} }

One of the most important consequences of Lemma~\ref{lemma:size-3}, when applied to the above setup, is the following structural characterisation of $\mathcal F$. It tells us that the Hasse diagram of $\{X\in \mathcal{F}\backslash \{\emptyset\} : X\subseteq A \text{ for some } A\in \mathcal{A}\}$ is a disjoint union of rooted trees, with each root drawn at the top.

\begin{lemma}\label{lemma:forest}
Let $A\in\mathcal A$ and $\mathcal{T}_A = \{X\in \mathcal{F}:\emptyset\neq X\subseteq A\}$. Then $\mathcal {T}_A$ is rooted tree with root $A$. Moreover, for every $W\in \mathcal{T}_A\setminus \{A\}$, the set $\mathcal{U}_W= \{X\in\mathcal{F}:W\subsetneq X\}$ has width at most 2.
\end{lemma}

\begin{proof}
First, suppose that there exist $X_1, X_2, X_3\in\mathcal T_A$ such that $X_3\subseteq X_1\cap X_2$, and $X_1$ and $X_2$ are incomparable. This is depicted below.

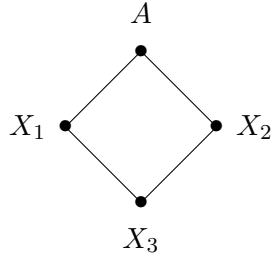
\begin{figure}[h]
\centering
\begin{tikzpicture}
\node at (0,0) {$\bullet$};
\node at (1,1) {$\bullet$};
\node at (-1,1) {$\bullet$};
\node at (0,2) {$\bullet$};
\draw (0,0) -- (1,1) -- (0,2) -- (-1,1) -- (0,0);
\node at (0,-0.5) {$X_3$};
\node at (0,2.5) {$A$};
\node at (1.5,1) {$X_2$};
\node at (-1.5,1) {$X_1$};
\end{tikzpicture}
\end{figure}

Since $\mathcal A$ has size at least 3, there exists $A^\prime\neq A$, $A^\prime\in \mathcal{A}$. By Lemma \ref{lemma:size-3} and the fact that we are inside $\mathcal T_A$, we get that $X_1$, $X_2$, and $X_3$ are incomparable to $A^\prime$, and so the four of them form a copy of $\widehat{\mathcal{V}}$ in $\mathcal F$, a contradiction. Therefore $\mathcal{T}_A$ is indeed a tree.

For the last part, let us fix $W\in \mathcal{T}_A\setminus\{A\}$. We note that if $X\in \mathcal{U}_W\setminus \mathcal{U}$, then $W\subsetneq X$, and $X$ is not above any element of $\mathcal{A}$. Therefore, by the maximality of $\mathcal A$, there exists $A^\prime\in \mathcal{A}$ such that $X\subseteq A^\prime$. We must have that $A^\prime = A$, otherwise $W$ will be below two elements of $\mathcal{A}$, contradicting Lemma~\ref{lemma:size-3}.

Putting everything together we get that all members of $\mathcal{U}_W\setminus\mathcal{U}$ lie below $A$ and, by definition, above $W$. Since we already established that $\mathcal T_A$ is a tree, we deduce that $\mathcal U_W\setminus\mathcal U$ is a chain.

Finally, suppose that we have 3 sets $Y_1,Y_2,Y_3\in\mathcal{U}_W$ that are pairwise incomparable. Since $\mathcal{U}_W\setminus\mathcal{U}$ is a chain and Lemma~\ref{lemma:antichain} tells us that the width of $\mathcal U$ is at most 2, we get that exactly two sets out of $Y_1,Y_2,Y_3$ are in $\mathcal{U}$. Without loss of generality, we may assume that $Y_3\not\in \mathcal{U}$. This, combined with the fact that $W\subseteq Y_3$ gives, by the same argument as above, that $Y_3\subseteq A$. However, $Y_3$ is incomparable to $Y_1$ and $Y_2$, which are in $\mathcal U$, implying that $A$ is incomparable to both $Y_1$ and $Y_2$. Therefore $A,Y_1, Y_2$ form an antichain of size 3 in $\mathcal A\cup\mathcal U$ that is not totally contained within $\mathcal{A}$, contradicting Lemma~\ref{lemma:antichain}.
\end{proof}

Now, returning to our setup, let $\mathcal{S}$ be the set of minimal elements of $\mathcal{U}$, which is non-empty as $[n]\in\mathcal{U}$. Since $\mathcal{S}$ is an antichain in $\mathcal{U}$, Lemma~\ref{lemma:antichain} tells us that $|S|\in \{1,2\}$, hence our analysis naturally splits.
\vspace{1em}

\noindent\underline{\textbf{Case I.} $|S|=2$.} Let $\mathcal{S} = \{X,Y\}$. We first observe that we cannot have every element of $\mathcal{A}$ contained in both $X$ and $Y$. Indeed, if that were the case, then $X\cap Y$ would be a separator of $\mathcal{F}$. Firstly, $X\cap Y\notin\{\emptyset, [n]\}$ since it contains all elements of $\mathcal A$, so it is not $\emptyset$, and it is strictly below both $X$ and $Y$, so it is not $[n]$. Next, every element of $\mathcal{U}$ contains at least one of $X$ and $Y$, and hence $X\cap Y$, every element of $\mathcal A$ is contained in both, hence in $X\cap Y$, and for a set outside of $\mathcal{A}\cup \mathcal U$, by the maximality of $\mathcal{A}$, it  sits below a set of $\mathcal{A}$, and therefore below both $X$ and $Y$, and thus below $X\cap Y$.

Therefore, without loss of generality, there exists $\overline{A}\in\mathcal{A}$ such that $\overline{A}\subset X$, and $\overline{A}$ is incomparable to $Y$. Moreover, if $A\in \mathcal{A}\setminus\{\overline{A}\}$, then by Lemma~\ref{lemma:antichain} $A,\overline{A},Y$ cannot form an antichain, thus $A\subset Y$.

As previously discussed, $\mathcal U$ may be covered by two chains in $\mathcal F$. Therefore $X$ is above a chain, which we can make maximal in the sense that there exists no element of $\mathcal F$ between consecutive terms of the chain. We denote this chain by $\mathcal C_X$. Therefore, the other part that covers $\mathcal{U}$ must be a chain above $Y$, which we call $\mathcal{C}_Y$. For clarity, let $\mathcal C_X=\{X_1,\dots,X_m\}$, where
$$\overline{A} = X_0\subset X_1 = X\subset X_2\subset \dots \subset X_m = [n],$$
with no member of $\mathcal{F}$ strictly between consecutive terms $X_i\subset X_{i+1}$ for all $0\leq i\leq m-1$.

The next lemma is one of our three main counting tools to achieve the claimed linear lower bound.

\begin{lemma}\label{lemma:count1} We have that $|\mathcal{F}|\geq n-|\overline{A}|+1.$
\end{lemma}
\begin{proof}
We will build an injection from $[n]\setminus\overline{A}$ into $\mathcal{F}\setminus\{\emptyset\}$. Therefore, fix $i\not\in\overline{A}$. Then, by construction, there exists $0\leq k\leq m-1$ such that $i\not\in X_k$, but $i\in X_{k+1}$. If $X_{k+1} = X_k\cup\{i\}$, then we send $i\mapsto X_{k+1}$. Otherwise, $X_k\cup\{i\}\not\in\mathcal{F}$, as it would be strictly between $X_k$ and $X_{k+1}$. Hence, $\mathcal{F}\cup \{X_k\cup\{i\}\}$ contains a $\widehat{\mathcal{V}}$ in which $X_k\cup\{i\}$ is one of the vertices. In this case, we claim the following.
\begin{claim1} There exists $Q_i\in\mathcal{F}$ such that $Q_i\setminus X_k = \{i\}$.
\end{claim1}
Once we have established the claim, the injection is complete as in this case we send $i\mapsto Q_i$. Note of course that no $Q_i$ can be equal to $\emptyset$.
\begin{proof}
Since $\mathcal F\cup\{X_k\cup\{i\}\}$ contains a $\widehat{\mathcal V}$, we consider 3 cases, depending on where $X_k\cup\{i\}$ is in the $\widehat{\mathcal{V}}$.

\noindent\textbf{Case 1.} $X_k\cup\{i\}$ is at the bottom of the $\mathcal{V}$.

More precisely, there exist $R,S,T\in\mathcal F$ such that $R,S, T$ form an antichain, $X_k\cup\{i\}\subset R,S$, and $T$ and $X_k\cup\{i\}$ are incomparable, as depicted below.

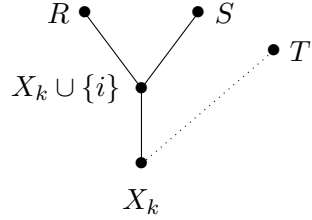
\begin{figure}[h]
\centering
\begin{tikzpicture}
\node at (0,0) {$\bullet$};
\node at (0,-0.5) {$X_k$};
\draw (0,0) -- (0,1);
\node at (0,1) {$\bullet$};
\draw (0,1) -- (0.75,2);
\draw (0,1) -- (-0.75,2);
\node at (0.75,2) {$\bullet$};
\node at (-0.75,2) {$\bullet$};
\node at (-1.1,2) {$R$};
\node at (1.1,2) {$S$};
\node at (-1,1) {$X_k\cup\{i\}$};
\draw[dotted] (0,0) -- (1.75,1.5);
\node at (1.75,1.5) {$\bullet$};
\node at (2.1,1.5) {$T$};
\end{tikzpicture}
\end{figure}

However, $R,S,T, X_k$ cannot form a $\widehat{\mathcal V}$ as they are all in $\mathcal F$. Thus $X_k$ and $T$ must be comparable, and hence $X_k\subset T$, as otherwise $T\subseteq X_k\cup\{i\}$, a contradiction. Since by construction $\overline{A}\subseteq X_k$, we note that $R,S,T$ form an antichain of size $3$ in $\mathcal{U}$, a contradiction. Therefore this case is impossible.

\noindent\textbf{Case 2.} $X_k\cup\{i\}$ is the isolated point.

In this case, there exist $R,S,T\in\mathcal F$ such that they are all incomparable to $X_k\cup\{i\}$, $R$ and $S$ are incomparable, and $T\subset R,S$, as illustrated below.

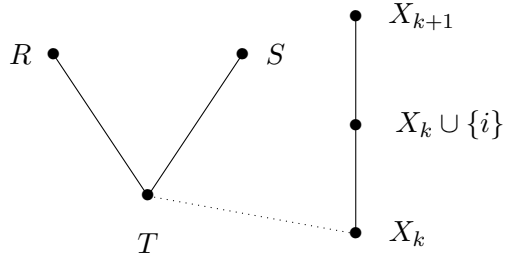
\begin{figure}[h]
\centering
\begin{tikzpicture}[xscale=1.25, yscale=1.25]
\node at (0,0) {$\bullet$};
\node at (0,-0.5) {$T$};
\draw (0,0) -- (1,1.5);
\draw (0,0) -- (-1,1.5);
\node at (1,1.5) {$\bullet$};
\node at (-1,1.5) {$\bullet$};
\node at (-1.35,1.5) {$R$};
\node at (1.35,1.5) {$S$};
\node at (2.2,0.75) {$\bullet$};
\node at (3.2,0.75) {$X_k\cup \{i\}$};
\draw (2.2,-0.4) -- (2.2,1.9);
\node at (2.2,-0.4) {$\bullet$};
\node at (2.2,1.9) {$\bullet$};
\node at (2.75,-0.4) {$X_k$};
\node at (2.9,1.9) {$X_{k+1}$};
\draw[dotted] (0,0) -- (2.2,-0.4);
\end{tikzpicture}
\end{figure}

In order for $R,S,T,X_k$ to not form a $\widehat{\mathcal V}$ in $\mathcal F$, we must have $X_k$ strictly below one of $R,S$ and $T$. If $X_k\subsetneq T$, then $R,S\in\mathcal{U}$. Moreover, note that $R$ and $S$ are incomparable to $X_{k+1}$. Indeed, if for example, $S$ is comparable to $X_{k+1}$, since it is incomparable to $X_k\cup\{i\}\subset X_{k+1}$, we must have $S\subset X_{k+1}$. This implies that $X_k\subset S\subset X_{k+1}$, contradicting the maximality of $\mathcal C_X$ inside $\mathcal F$. Thus, $R$ and $S$ are on the other chain $\mathcal{C}_Y$, a contradiction as they are incomparable.

The above argument shows in fact that $X_k$ cannot be below both $R$ and $S$. Therefore, without loss of generality, we may assume that $X_k$ is below $S$, and $X_k$ and $R$ are incomparable. As argued above, we have that $S\in \mathcal{U}$ and $S$ and $X_{k+1}$ are incomparable, which implies that $S$ is on the other chain $\mathcal C_{Y}$, and consequently $Y\subseteq S$. We now have the following diagram.

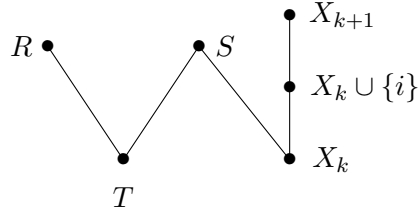
\begin{figure}[h]
\centering
\begin{tikzpicture}
\node at (0,0) {$\bullet$};
\node at (0,-0.5) {$T$};
\draw (0,0) -- (1,1.5);
\draw (0,0) -- (-1,1.5);
\node at (1,1.5) {$\bullet$};
\node at (-1,1.5) {$\bullet$};
\node at (-1.35,1.5) {$R$};
\node at (1.35,1.5) {$S$};
\node at (2.2,0.95) {$\bullet$};
\node at (3.2,0.95) {$X_k\cup \{i\}$};
\draw (2.2,0) -- (2.2,1.9);
\node at (2.2,0) {$\bullet$};
\node at (2.2,1.9) {$\bullet$};
\node at (2.75,0) {$X_k$};
\node at (2.9,1.9) {$X_{k+1}$};
\draw (1,1.5) -- (2.2,0);
\end{tikzpicture}
\end{figure}

Next, we observe that $R\not\in \mathcal{U}$. This is because if $R\in\mathcal U$, then $R$ has to be on one of the two chains that covers it. However, since $R$ and $S$ are incomparable, and $S\in\mathcal C_{Y}$, we have that $R\notin\mathcal C_{Y}$ which implies that $R\in \mathcal{C}_X$, forcing $R$ to be comparable to $X_k$, a contradiction.

Therefore $R\notin\mathcal U$, and by maximality of $\mathcal{A}$, $R\subseteq A$ for some $A\in\mathcal{A}$. If $A\neq\overline{A}$, then $A$ and $Y$ are comparable (otherwise $A,\overline{A}$ and $Y$ form an antichain in $\mathcal A\cup\mathcal U$ that is not entirely contained in $\mathcal A$), and so $R\subseteq A\subseteq Y\subseteq S$, a contradiction. Thus $A=\overline{A}$. Then, $R\subseteq \overline{A}\subseteq X_k\cup\{i\}$, a contradiction. Therefore this case is impossible too.

\noindent\textbf{Case 3:} $X_k\cup\{i\}$ is at the top of the $\mathcal{V}$.

Therefore there exist $R,S,T\in\mathcal F$ such that $S\subset X_k\cup\{i\}$, $S\subset R$, $R$ and $X_k\cup\{i\}$ are incomparable, and $T$ is incomparable to all $R,S,X_k\cup\{i\}$. This is depicted below.
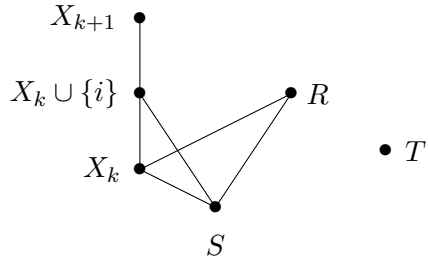
\begin{figure}[h]
\centering
\begin{tikzpicture}
\node at (0,0) {$\bullet$};
\node at (0,-0.5) {$S$};
\draw (0,0) -- (1,1.5);
\node at (1,1.5) {$\bullet$};
\draw (0,0) -- (-1,1.5);
\node at (-1,1.5) {$\bullet$};
\node at (-2,1.5) {$X_{k}\cup\{i\}$};
\node at (1.35,1.5) {$R$};
\node at (-1,0.5) {$\bullet$};
\draw (0,0) -- (-1,0.5);
\draw (-1,0.5) -- (-1,2.5);
\draw node at (-1,2.5) {$\bullet$};
\draw node at (-1.5,0.5) {$X_k$};
\draw node at (-1.75,2.5) {$X_{k+1}$};
\draw (-1,0.5) -- (1,1.5);
\draw node at (2.25,0.75) {$\bullet$};
\draw node at (2.65,0.75) {$T$};
\end{tikzpicture}
\end{figure}

If $X_k$ and $S$ are incomparable, then $S\subset X_k\cup\{i\}$ implies $i\in S$ and consequently $S\setminus X_k = \{i\}$ as desired. Therefore we may assume that $X_k$ and $S$ are comparable. If $X_k\subset S$, then $|X_k|<|S|<|X_{k}\cup\{i\}|=|X_k|+1$, which is impossible. Thus $S\subseteq X_k$.

Suppose that $X_k\neq S$. Since $X_k,S,R,T$ cannot form a $\widehat{\mathcal V}$, we must have $X_k$ comparable to at least one of $R$ or $T$. However, $X_k$ and $T$ are incomparable since $X_k$ is between $S$ and $X_k\cup\{i\}$, both of which are incomparable to $T$. Therefore $X_k\subset R$, and so $R\in\mathcal{U}$. As in the previous case, $R$ and $X_{k+1}$ are incomparable (there is no element of $\mathcal F$ between $X_k$ and $X_{k+1}$), thus $R\notin \mathcal C_X$, and so $R\in \mathcal{C}_Y$, which implies that $Y\subseteq R$. 

Next, if $T\in\mathcal U$, since $R$ and $T$ are incomparable and $R\in\mathcal C_Y$ , we must have that $T\in\mathcal{C}_X$, which contradicts the fact that $X_k$ is incomparable to $T$. Thus $T\not\in\mathcal{U}$. Therefore, there exists $A\in\mathcal A$ such that $T\subseteq A$. If $A\neq\overline{A}$, then $A\subset Y$, and hence $T\subset Y\subseteq R$, a contradiction. If $A =\overline{ A}$, then $T\subseteq A\subseteq X\subseteq X_k$, a contradiction again.

Therefore we must have $X_k = S$, and the figure then becomes the following.

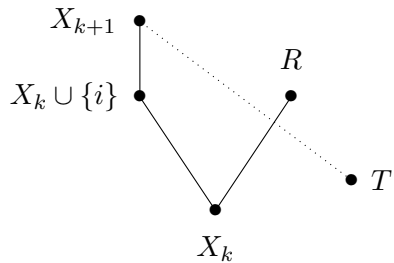
\begin{figure}[h]
\centering
\begin{tikzpicture}
\node at (0,0) {$\bullet$};
\node at (0,-0.5) {$X_k$};
\draw (0,0) -- (1,1.5);
\draw (0,0) -- (-1,1.5);
\node at (1,1.5) {$\bullet$};
\node at (-1,1.5) {$\bullet$};
\node at (1,2) {$R$};
\node at (-2,1.5) {$X_{k}\cup\{i\}$};
\draw (-1,1.5) -- (-1,2.5);
\node at (-1,2.5) {$\bullet$};
\node at (-1.75,2.5) {$X_{k+1}$};
\node at (1.8,0.4) {$\bullet$};
\node at (2.2,0.4) {$T$};
\draw[dotted] (1.8,0.4) -- (-1,2.5);
\end{tikzpicture}
\end{figure}

By the same argument as above, $R$ is incomparable to $X_{k+1}$, and since $X_k\subset R$, we have that $R\in\mathcal U$ and so $R\in\mathcal{C}_Y$. Furthermore, $X_{k+1},X_k, R, T$ cannot form a $\widehat{\mathcal V}$, $T$ and $X_{k+1}$ must be comparable. Since $X_k\subset X_{k+1}$, and $T$ and $X_{k}$ are incomparable, we have that $T\subset X_{k+1}$. If $T\in\mathcal{U}$, since $R$ and $T$ are incomparable, we must have that $T\in\mathcal{C}_X$, a contradiction as $T$ and $X_k$ are incomparable. If $T\not\in \mathcal{U}$, then there exists $A\in\mathcal A$ such that $T\subseteq A$. If $A\neq\overline{A}$, then $T\subset Y\subseteq R$, a contradiction. If $A=\overline{A}$, then $T\subseteq A\subset X\subseteq X_k$, a contradiction again.
\end{proof}
Therefore, for every $i\not\in \overline{A}$, we get that either $X_k\cup\{i\}\in \mathcal{F}$, or $X_k\cup\{i\}\notin\mathcal F$ in which case there exists $Q_i\in\mathcal F$ such that $Q_i\setminus X_k =\{i\}$. We recall that $k$ is such that $i\notin X_k$ but $i\in X_{k+1}$. We claim that the function sending $i$ to $X_{k}\cup\{i\}$ in the first case, and to $Q_i$ in the second, is an injection from $[n]\setminus\overline A$ to $\mathcal F$, which finishes the proof.

Indeed, let $i\neq j$ be two elements not in $\overline{A}$. Let $k_1$ be such that $i\in X_{k_1+1}\setminus X_{k_1}$ and $j\in X_{k_2+1}\setminus X_{k_2}$. Suppose that $i$ and $j$ are assigned the same set in $\mathcal F$. We cannot have $X_{k_1}\cup\{i\} = X_{k_2}\cup \{j\}$ as this would mean that $X_{k_1}$ and $X_{k_2}$ are incomparable, a contradiction. Therefore at least one of $X_{k_1}\cup\{i\}$ and $X_{k_2}\cup\{j\}$ is not in $\mathcal F$.

Suppose now that neither set is in $\mathcal F$, and $Q_i = Q_j = Q$. By construction we have that $Q\subset X_{k_1}\cup\{i\}$, $Q\subset X_{k_2}\cup\{j\}$, and that $i,j\in Q$. Without loss of generality, we may assume that $k_1\leq k_2$. Since $X_{k_1}\subseteq X_{k_2}$, and $j\not\in X_{k_2}$ we get that $j\not\in X_{k_1}$. However, $j\in Q\subset X_{k_1}\cup\{i\}$, which implies that $j\in X_{k_1}$, a contradiction.

Finally, suppose that $X_{k_1}\cup\{i\}\in\mathcal F$, $X_{k_2}\cup\{j\}\notin\mathcal F$, and there exists $Q_j\in\mathcal F$ such that $Q_j \subset X_{k_2}\cup \{j\}$ and $j\in Q_j$. If $X_{k_1}\cup\{i\}=Q_j$, then $j\in X_{k_1}$, and hence $X_{k_2}\subseteq X_{k_1}$. This in turn implies that $i\not\in X_{k_2}$. However, $X_{k_1}\cup\{i\}=Q_j\subset X_{k_2}\cup\{j\}$, a contradiction.

Thus, the claimed map is indeed an injection, and so $|\mathcal{F}\setminus\{\emptyset\}|\geq n-|\overline{A}|$, which completes the proof.
\end{proof}
The next lemma is our second main counting tool. Here we look at the singletons inside $\overline{A}$ and inject a certain amount of them into $\mathcal F$, this time going down the rooted tree $\mathcal T_{\overline{A}}$.
\begin{lemma}\label{lemma:count2} Suppose that there exists $C\in\mathcal{T}_{\overline{A}}$ non-minimal. Then, there exists $B\subset C$ such that no member of $\mathcal{F}$ is between $B$ and $C$, and for every $i\in C\setminus B$ there exists $Z_i\in\mathcal{F}$ such that $Z_i\setminus B= \{i\}$.
\end{lemma}

\begin{proof}
Suppose that such a set does not exist. Let $B\in\mathcal F$ be such that $C$ covers $B$ in $\mathcal F$. Then there exists $i\in C\setminus B$ such that there exists no $Z_i\in\mathcal F$ with the property that $Z_i\setminus B=\{i\}$. We must then have $|C\setminus B|\geq 2$, as otherwise $C$ itself would satisfy $|C\setminus B| = 1$. Therefore, $B\subset B\cup\{i\}\subset C$, and so $B\cup\{i\}\not\in \mathcal{F}$. By saturation, adding $B\cup\{i\}$ to $\mathcal{F}$ creates a copy of $\widehat{\mathcal{V}}$, and there are then 3 cases depending on where $B\cup\{i\}$ is in this copy.

\noindent\textbf{Case 1:} $B\cup\{i\}$ is at the bottom of the $\mathcal{V}$.

Therefore there exist $R,S,T\in\mathcal F$ such that $T$ is incomparable to all of $R,S,B\cup\{i\}$, $B\cup\{i\}\subset R,S$, and $R$ and $S$ are incomparable. This is pictured below.

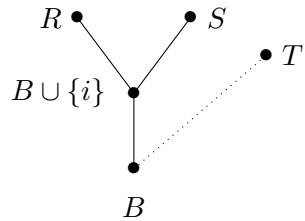
\begin{figure}[h]
\centering
\begin{tikzpicture}
\node at (0,0) {$\bullet$};
\node at (0,-0.5) {$B$};
\draw (0,0) -- (0,1);
\node at (0,1) {$\bullet$};
\draw (0,1) -- (0.75,2);
\draw (0,1) -- (-0.75,2);
\node at (0.75,2) {$\bullet$};
\node at (-0.75,2) {$\bullet$};
\node at (-1.1,2) {$R$};
\node at (1.1,2) {$S$};
\node at (-1,1) {$B\cup\{i\}$};
\draw[dotted] (0,0) -- (1.75,1.5);
\node at (1.75,1.5) {$\bullet$};
\node at (2.1,1.5) {$T$};
\end{tikzpicture}
\end{figure}

In order for $B,R,S,T$ to not form a $\widehat{\mathcal V}$ in $\mathcal F$, we need $B$ and $T$ to be comparable. Since $B\cup\{i\}$ and $T$ are incomparable, we must have $B\subset T$. But then $\mathcal{U}_B$ contains the antichain $R,S,T$, which contradicts the second part of Lemma \ref{lemma:forest}, since $B\in\mathcal T_{\overline{A}}$.

\noindent\textbf{Case 2:} $B\cup\{i\}$ is the isolated point.

Therefore there exist $R,S,T\in\mathcal F$ such that $R,S,T$ form a $\mathcal V$ with $T$ the minimal element of $\mathcal V$, and $B\cup\{i\}$ incomparable to all $R,S,T$, as depicted below. 

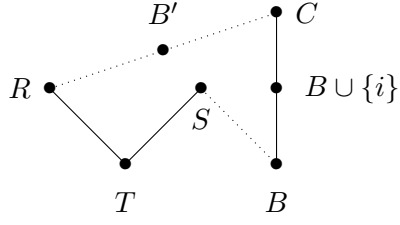
\begin{figure}[h]
\centering
\begin{tikzpicture}
\node at (0,0) {$\bullet$};
\node at (1,1) {$\bullet$};
\node at (-1,1) {$\bullet$};
\node at (2,0) {$\bullet$};
\node at (2,1) {$\bullet$};
\node at (2,2) {$\bullet$};
\node at (0.5,1.5) {$\bullet$};
\draw (-1,1) -- (0,0) -- (1,1);
\draw (2,0) -- (2,1) -- (2,2);
\draw [dotted] (1,1) -- (2,0);
\draw [dotted] (-1,1) -- (2,2);
\node at (0,-0.5) {$T$};
\node at (2,-0.5) {$B$};
\node at (3,1) {$B\cup\{i\}$};
\node at (2.4,2) {$C$};
\node at (-1.4,1) {$R$};
\node at (0.5,2) {$B^\prime$};
\node at (1,0.6) {$S$};
\end{tikzpicture}
\end{figure}

Then, $R,S,T,B$ cannot form a $\widehat{\mathcal V}$, and so one of $R,S,T$ must be comparable to $B$. Since they are all incomparable to $B\cup\{i\}$, $B$ must be strictly below at least one of them.  If $B$ is below both $R$ and $S$, then $C$ is incomparable to both $R$ and $S$. Indeed, if for example $S$ and $C$ were comparable, since $B\cup\{i\}\subset C$, we must have $S\subset C$. However, this contradicts the fact that $C$ covers $B$ in $\mathcal F$. Therefore $R,S,C$ form an antichain of length 3 in $\mathcal{U}_B$, contradicting Lemma~\ref{lemma:forest}.
    
Therefore $B$ is incomparable to exactly one of $R$ and $S$, and consequently to $T$. Without loss of generality we may assume that $B$ is incomparable to $R$, and that $B\subset S$. As above, this also implies that $S$ is incomparable to $C$.

Since $R,S,B,C$ cannot form a copy of $\widehat{\mathcal{V}}$, we must have $R$ and $C$ comparable. Since $B\cup\{i\}$ is incomparable to $R$ and a subset of $C$, we must have $R\subset C$. Now let $R\subseteq B^\prime\subset C$ be covered by $C$ in $\mathcal{F}$. Note that this implies that $B^\prime\neq B$. Moreover, we cannot have $S\subseteq B^\prime$, as otherwise $S$ and $C$ would be comparable, a contradiction, and if $B^\prime\subseteq S$, then $R$ and $S$ would be comparable, a contradiction again. Hence, we must have that $B^\prime$ and $S$ are incomparable. 

We now focus on where $S$ is in the family. Since $B\subset S$ and $B\in\mathcal T_{\overline{A}}$, we have that $S\in\mathcal T_{\overline{A}}\cup\mathcal A\cup\mathcal U.$ We note that $S\not\in \mathcal{T}_{\overline A}$, as otherwise we would have the following forbidden configuration by Lemma~\ref{lemma:forest}.

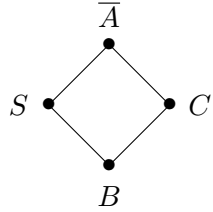
\begin{figure}[h]
\centering
\begin{tikzpicture}[xscale=0.8, yscale=0.8]
\node at (0,0) {$\bullet$};
\node at (1,1) {$\bullet$};
\node at (-1,1) {$\bullet$};
\node at (0,2) {$\bullet$};
\draw (0,0) -- (1,1) -- (0,2) -- (-1,1) -- (0,0);
\node at (0,-0.5) {$B$};
\node at (0,2.5) {$\overline{A}$};
\node at (1.5,1) {$C$};
\node at (-1.5,1) {$S$};
\end{tikzpicture}
\end{figure}

Furthermore, $\overline{A}\not\subseteq S$ as $B^\prime\subset\overline{A}$ and $B'$ is incomparable to $S$. Therefore $\overline{A}$ and $S$ are incomparable, which implies that $S\notin\mathcal A$, as otherwise $B$ would be below at least $2$ elements of $\mathcal{A}$. This implies that $S\in \mathcal U$, which coupled with the fact that $S$ and $\overline{A}$ are incomparable, gives that $S\in\mathcal C_Y$.

So in this case we produced another $B'$ covered by $C$, and a set $S\in\mathcal C_Y$ such that $B'$ and $S$ are incomparable, but $B\subset S$.
    
\noindent\textbf{Case 3:} $B\cup\{i\}$ is one of the two top elements of the $\mathcal{V}$.

In other words, there exist $R,S,T\in\mathcal F$ such that $R\subset B\cup\{i\}$ and $R\subset T$, while $T$ and $B\cup\{i\}$ are incomparable, and $S$ is incomparable to all of $R,T,B\cup\{i\}$. This is pictured below. 

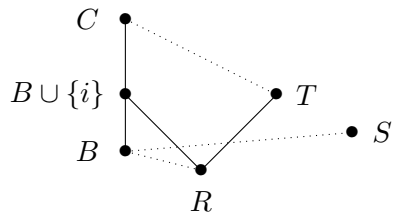
\begin{figure}[h]
\centering
\begin{tikzpicture}
\node at (0,0) {$\bullet$};
\node at (-1,1) {$\bullet$};
\node at (1,1) {$\bullet$};
\node at (2,0.5) {$\bullet$};
\node at (-1,0.25) {$\bullet$};
\node at (-1,2) {$\bullet$};
\draw (-1,1) -- (0,0) -- (1,1);
\draw (-1,0.25) -- (-1,2);
\draw [dotted] (0,0) -- (-1,0.25);
\draw [dotted] (-1,0.25) -- (2,0.5);
\draw [dotted] (-1,2) -- (1,1);
\node at (0,-0.4) {$R$};
\node at (-1.5, 0.25) {$B$};
\node at (-1.9,1) {$B\cup\{i\}$};
\node at (-1.5,2) {$C$};
\node at (1.4,1) {$T$};
\node at (2.4,0.5) {$S$};
\end{tikzpicture}
\end{figure}

If $B$ and $R$ are incomparable, since $R\subset B\cup\{i\}$, we must have $i\in R$, and so $R\setminus B = \{i\}$, contradicting the choice of $i$. Thus, we must have that $R\subseteq B$. We note that this immediately implies that $B$ and $S$ are incomparable. Indeed, $B\subseteq S$ gives $R\subset S$, a contradiction, and $S\subseteq B$ gives $S\subset B\cup\{i\}$, a contradiction again.

Next, suppose that $B$ and $T$ are incomparable. In particular, since $R\subset T$, this implies that $R\neq B$. Therefore, since $B,R,T,S$ form a $\widehat{\mathcal{V}}$ in $\mathcal F$, we get a contradiction. 

Therefore, $B$ and $T$ must be comparable. We cannot have $T\subseteq B$ as this would imply $T\subset B\cup\{i\}$, a contradiction. Hence, we have that $B\subset T$. Looking now at the relation between $C$ and $T$ we see that they must be incomparable. Indeed, if they were comparable, since $B\cup\{i\}\subset C$ and $B\cup\{i\}$ and $T$ are incomparable, we can only have $T\subset C$. However, $T$ would now be strictly between $B$ and $C$, contradicting the fact that $C$ covers $B$. Therefore, since now $C,B,T,S$ cannot form a $\widehat{\mathcal V}$ in $\mathcal F$, we must have $S$ comparable to $C$. Since $B\cup\{i\}\subset C$ and $B\cup\{i\}$ and $S$ are incomparable, we must have $S\subset C$. The picture then becomes the following.

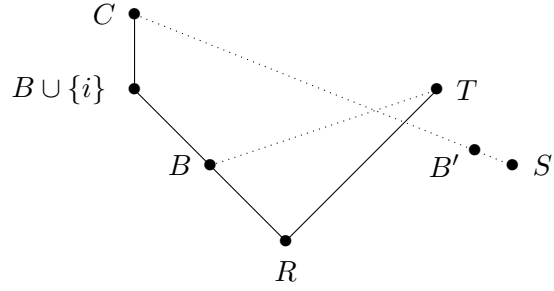
\begin{figure}[h]
\centering
\begin{tikzpicture}
\node at (0,0) {$\bullet$};
\node at (-1,1) {$\bullet$};
\node at (-2,2) {$\bullet$};
\node at (2,2) {$\bullet$};
\node at (3,1) {$\bullet$};
\node at (-2,3) {$\bullet$};
\draw (2,2) -- (0,0) -- (-2,2) -- (-2,3);
\draw [dotted] (-2,3) -- (3,1);
\node at (2.5,1.2) {$\bullet$};
\node at (2.1,1) {$B^\prime$};
\node at (3.4,1) {$S$};
\node at (0,-0.4) {$R$};
\node at (-1.4,1) {$B$};
\node at (-3,2) {$B\cup\{i\}$};
\node at (-2.4,3) {$C$};
\node at (2.4,2) {$T$};
\draw [dotted] (-1,1) -- (2,2);
\end{tikzpicture}
\end{figure}

With the same argument as above, $C$ covering $B$ implies that $B$ and $S$ are incomparable. Let $B'\in\mathcal F$ be such that $S\subseteq B^\prime$ and $C$ covers $B^\prime$. Since $B$ and $S$ are incomparable, we get that $B\neq B^\prime$. Moreover, since $T$ is incomparable to both $C$ and $S$ and $S\subseteq B^\prime\subset C$, we get that $T$ and $B^\prime$ are incomparable.

Finally, we establish where in the family $T$ is situated. Since $B\subset T$, we must have $T\in\mathcal T_{\overline{A}}\cup\mathcal A\cup\mathcal U$. Since $C$ is incomparable to $T$ and $C\subseteq\overline{A}$, we cannot have $\overline{A}\subseteq T$. Also, as $C,B,T$ form a $\mathcal{V}$, we cannot have $T\subset\overline{A}$, otherwise we contradict the fact that $\mathcal T_{\overline{A}}$ is a tree. Hence $T\notin\mathcal T_{\overline{A}}$ and $T$ is incomparable to $\overline{A}$. As before, all of this, combined with the fact that $B\subset\overline{A}$ and $B\subset T$, implies that $T\in \mathcal{U}$. Since $T$ and $\overline{A}$ are incomparable, we must have $T\in \mathcal{C}_Y$.

We conclude that for every $B\in\mathcal F$ that is covered by $C$, we find a different $B^\prime$, also covered by $C$, and a set $M_B\in\mathcal{C}_Y$ such that $B\subset M_B$, and $M_B$ incomparable to $B^\prime$. Performing the same as above for $B^\prime$, we get $B^{\prime\prime}$ and $M_{B^\prime}\in \mathcal{C}_Y$ with $B^\prime\subsetneq M_{B^\prime}$, and $M_{B^\prime}$ is incomparable to $B^{\prime\prime}$.

We cannot have $M_{B^\prime}\subseteq M_B$ as this would imply $B^\prime\subset M_B$, a contradiction. Hence we must have $M_B\subset M_{B^\prime}$. Therefore, iterating this process, we get an infinite chain of $M_{B}$'s, a contradiction. This completes the proof.
\end{proof}
Putting together Lemma \ref{lemma:count2}, Lemma \ref{lemma:count1}, and Lemma \ref{lemma:minimal}, we arrive at the desired lower bound.

Indeed, starting with $C_0 =\overline{A}$, we apply Lemma \ref{lemma:count2} repeatedly until a minimal member of $\mathcal T_{\overline{A}}$ is reached, call it $M$. It is easy to see that $M$ is also a minimal element of $\mathcal F\setminus\{\emptyset\}$. Therefore we obtain a strictly decreasing chain, $\overline{A} = C_0\supset C_1\supset\dots\supset C_l = M$. By Lemma~\ref{lemma:count2} the chain is maximal in $\mathcal{F}$. Moreover, for every $i\in C_r\setminus C_{r+1}$, there exists $Z_i\in\mathcal{F}$ such that $Z_i\setminus C_{r+1} = \{i\}$. 

Moreover, the sets $Z_i$ for $i\in\overline{A}\setminus M$ are pairwise distinct. Indeed, suppose that $Z_i=Z_j=Z$ for some $i\neq j$. Then, they cannot come from the same `gap', i.e.~there exists $0\leq r<l$ such that $i,j\in C_r\setminus C_{r+1}$, as $Z_i$ and $Z_j$ would have different difference with $C_{r+1}$. Thus $i\in C_r\setminus C_{r+1}$ and $j\in C_s\setminus C_{s+1}$, and without loss of generality, we may assume that $r<s$. We have that $C_{s}\subseteq C_{r+1}$ and $j\in C_s\setminus C_{s+1}$, and so $j\in C_{r+1}$. hence $Z_j\subseteq C_{r+1}$. However, $Z\subset C_{s}\cup\{j\}\subseteq C_{r+1}$. This is a contradiction as $Z_i\setminus C_{r+1} = \{i\}$. Furthermore, $Z_i\in\mathcal F\setminus\{\emptyset\}$, and so  we get that
$$|\mathcal{F}| \geq |\overline{A}| - |M|+1.$$

Furthermore, looking at the elements that give the count in Lemma~\ref{lemma:count1}, they all contain a singleton outside $\overline{A}$, while the ones obtained from iterating Lemma~\ref{lemma:count2} are all subsets of $\overline{A}$, all different from $M$ (as they all contain at least one singleton outside $M$). So in fact, the two lemmas give $|\mathcal F\setminus\{\emptyset,M\}|\geq n-|M|$, and so $$|\mathcal F|\geq n-|M|+2.$$
However, Lemma \ref{lemma:minimal} gives that $|\mathcal{F}|\geq |M|+2$. Adding these together gives
$|\mathcal{F}|\geq \frac{n}{2}+2$, as claimed.

\vspace{1em}
\noindent\underline{\textbf{Case II.} $|S|=1$.} In this case $\mathcal U$ has exactly one minimal element, call it $N\in\mathcal F$. If $N$ is comparable to all elements of $\mathcal A$ (it has to be necessarily above all of them), then $N$ separates $\mathcal F$. Indeed, if $S\in\mathcal U$, then $N\subseteq S$, otherwise $S\subseteq A\subset N$ for some $A\in\mathcal A$. However, $\mathcal F$ is not separable, and since $\emptyset\neq N$, we must have $N=[n]$. This fact, together with Lemma~\ref{lemma:forest} implies that $\mathcal F\setminus\{\emptyset\}$ is a downward tree, rooted at $[n]$, hence it is $\mathcal V$-free. Moreover, if $S\notin \mathcal F\setminus\{\emptyset\}$, then adding $S$ to this family creates a $\mathcal V$. Indeed, if $S=\emptyset$ then, together with 2 elements of $\mathcal A$ we have a $\mathcal V$. Otherwise, $S\notin\mathcal F$ and so it forms a $\widehat{\mathcal V}$ with 3 sets of $\mathcal F$. Since none of these sets can be $\emptyset$, we conclude that $(\mathcal F\setminus\{\emptyset\})\cup\{S\}$ contains a $\widehat{\mathcal V}$, and hence a $\mathcal V$. Thus $\mathcal F\setminus\{\emptyset\}$ is $\mathcal V$-saturated, and so $|\mathcal F|-1\geq n+1$ \cite{ferrara2017saturation}, giving $|\mathcal F|\geq n+2>\frac{n}{2}+2$, as desired.

Therefore, we may assume that $N$ is not comparable to all elements of $\mathcal A$. Thus, there exists $\overline{A}\in\mathcal A$ such that $\overline{A}$ and $N$ are incomparable. Same as before, since we cannot have an antichain of size 3 in $\mathcal A\cup U$ that is not completely contained in $\mathcal A$, we get that all elements of $\mathcal A\setminus\{\overline{A}\}$ are below $N$. It follows that the set $U\cup\{\overline{A\}}$ has width at most 2. We cover this set with 2 chains, the same way we covered $\mathcal U$ in Case I. Let $\mathcal C_1$ be a saturated chain in $\mathcal F$ that starts at $\overline{A}$ and ends at $[n]$, and let $\mathcal C_2$ be the second chain, which will have to cover $N$. By exchanging the chains $\mathcal C_X$ and $\mathcal C_Y$ from Case I with the chains $\mathcal C_1$ and $\mathcal C_2$ respectively, and $Y$ with $N$, Lemmas~\ref{lemma:count1} and~\ref{lemma:count2} apply verbatim, giving the same $\frac{n}{2}+2$ lower bound.
\end{proof}

\section{Proof of the main theorem}

In this section, everything ties together to achieve linearity for the saturation number of $\widehat{\mathcal V}$. The linear upper bound is easy to come across, especially since one just needs a $\mathcal V$-saturated family to which $\emptyset$ is added. For the lower bound, Proposition~\ref{inseparable} completely settles the inseparable case, while a separable family naturally gives rise to an inductive argument using the inequality established at the end of Section 2.
\begin{theorem}\label{theorem:main} Let $n\geq2$. Then $$\frac{n}{2}+2\leq\text{sat}^*(n,\widehat{\mathcal V})\leq n+2.$$
\end{theorem}
\begin{proof}
We first establish the lower bound. Let $\mathcal F$ be a $\widehat{\mathcal V}$-saturated family with ground set $[n]$. If $n= 2$, then it is trivial to see that $\mathcal F$ has to contain the entire power set, and hence the inequality is true. Thus, we may suppose that $n\geq3$ and that the result is true for all $n_0<n$. Furthermore, if $|\mathcal F|\geq 2n$, then $|\mathcal F|\geq\frac{n}{2}+2$, so we may additionally assume that $|\mathcal F|<2n$. This means that the previous results apply. In particular, if $\mathcal F$ is not separable, then by Proposition~\ref{inseparable} we have that $|\mathcal F|\geq\frac{n}{2}+2$. If $\mathcal F$ is separable, let $S\notin\{\emptyset, [n]\}$ be a separator. Then $|\mathcal F|\geq\text{sat}^*(|S|,\widehat{\mathcal V})+\text{sat}^*(n-|S|,\widehat{\mathcal V})-1$. If $|S|\notin\{1,n-1\}$, applying the induction hypothesis we get that $|\mathcal F|\geq\frac{|S|}{2}+2+\frac{n-|S|}{2}+2-1$, thus $|\mathcal F|\geq\frac{n}{2}+3>\frac{n}{2}+2$. If, without loss of generality $|S|=1$, we get $|\mathcal F|\geq 2+\frac{n-1}{2}+2-1=\frac{n}{2}+2+\frac{1}{2}>\frac{n}{2}+2$. This finishes the induction step, and hence the proof of the lower bound.

For the upper bound, it is trivial to check that the family $\{\emptyset,[n],[n]\setminus\{i\}:i\in[n]\}$ is $\widehat{\mathcal V}$-saturated. Indeed, no copy of $\widehat{\mathcal V}$ exists (below any antichain there is only $\emptyset$), but if $X\notin\mathcal F$, let $i,j\notin X$, $i\neq j$, and $k\in X$. Then $X,[n]\setminus\{i\},[n]\setminus\{j\}, [n]\setminus\{k\}$ form a $\widehat{\mathcal V}$. This finishes the proof.
\end{proof}
\section{Concluding remarks}
It is worth pointing out that while we could guarantee the sets in Lemmas~\ref{lemma:count1} and~\ref{lemma:count2} to be pairwise distinct, we cannot guarantee them to be distinct from the ones in Lemma~\ref{lemma:minimal}, i.e. the witnesses $Q_i\in\mathcal F$ corresponding to the elements $i\in M$ where $M\setminus Q_i=\{i\}$. For example, it can happen that $M=\{i\}$, $Q_i=[n]\setminus\{i\}$, and $Q_i$ might be also associated to a singleton outside $\overline{A}$. A similar `collision' could in principle happen between Lemma~\ref{lemma:minimal} and Lemma~\ref{lemma:count2}.

If one could guarantee, without the empty set, $|M|+1$ sets in Lemma~\ref{lemma:minimal}, disjoint from all the ones in Lemma~\ref{lemma:count1} ($n-|\overline{A}|$), and disjoint from all the ones in Lemma~\ref{lemma:count2} (|$\overline{A}|-|M|)$. Then, in the inseparable case, together with the empty set, we would get the lower bound of exactly $n+2$, which also survives the induction presented in Theorem~\ref{theorem:main}. This would show that the \textit{exact} saturation number for $\widehat{\mathcal V}$ is $n+2$.

We strongly believe that this is true. It is possible that a more careful analysis of our methods yields this result. In particular, a promising strategy would be to show that for every counting `collision', more elements of the family can be counted, by either using the two chains above $\mathcal A$, or the tree structure underneath it.

Another interesting feature about $\widehat{\mathcal V}$-saturated is that not all of them contain sets $S,S\cup\{i\}\in\mathcal F$ with $i\notin S$ for all $i$ in the ground set. This is one of the main tools so far in showing a lower bound of $n+1$ -- this trivially follows from the acyclicity of the graph on vertices elements of $\mathcal F$, and edges between $S$ and $S\cup\{i\}$ for exactly one such pair per singleton $i$. Take for example $n=6$ and the following $\widehat{\mathcal V}$-saturated family which does not have such a pair for $i=3,5$: $$\mathcal F=\{\emptyset,\{1,3\},\{2,3\},\{1,4\},\{3,4\},\{1,5\},\{2,5\},\{3,5\},\{4,5\},\{2,6\},\{3,6\},\{5,6\},$$\vspace{-1.5em}$$\{1,3,4,5\},\{2,3,5,6\},\{1,2,3,5,6\},\{1,3,4,5,6\},\{1,2,3,4,5,6\}\}.$$ This example shows that Conjecture~\ref{conjecture:main} cannot be attacked in this way. However, the above construction is not minimal, as $n=6$ and $\mathcal F$ has $17>6+2$ sets. Thus, the following could be true.

\begin{conjecture}\label{conjconj}
Let $\mathcal P$ be such that $\text{sat}^*(n,\mathcal P)\to\infty$ as $n\to\infty$. Then, for $n$ large enough, if $\mathcal F\subseteq\mathcal P([n])$ is a $\mathcal P$-saturated family of size  $\text{sat}^*(n,\mathcal P)$, for every $i\in[n]$, there exists $S,S\cup\{i\}\in\mathcal F$, where $i\notin S$.
\end{conjecture}
This would of course imply that for any poset $\mathcal P$, the saturation number is either bounded or at least $n+1$, which is exactly Conjecture~\ref{conjecture:main}.

Lastly, it is worth asking the following question: given a poset $\mathcal P$, how closely related, \textit{structurally}, a saturated family for $\mathcal P$ is, to a saturated family for $\mathcal P$ plus an isolated point? Looking back on our proof for $\widehat{\mathcal V}$, our family is, more or less, a carefully chosen antichain $\mathcal A$, at most 2 chains above it, and a forest underneath it, the latter being exactly the structure of a $\mathcal V$-free family. Moreover, this structure was absolutely essential to proving the linear lower bound, as Lemma~\ref{lemma:count2} heavily relies on it.
\bibliographystyle{amsplain}
\bibliography{references}
\Addresses
\end{document}